\documentclass[12pt,reqno]{amsart}

\usepackage{amsmath}
\usepackage{amsthm}
\usepackage{amssymb}
\usepackage{hyperref}
\usepackage{a4wide}
\usepackage{calrsfs}

\allowdisplaybreaks

\numberwithin{equation}{section}

\newtheorem{theorem}{Theorem}[section]
\newtheorem{lemma}[theorem]{Lemma}

\newtheorem{corollary}[theorem]{Corollary}

\theoremstyle{remark}

\theoremstyle{remark}

\DeclareMathOperator{\ad}{ad}
\DeclareMathOperator{\Ad}{Ad}
\DeclareMathOperator{\Aut}{Aut}

\DeclareMathOperator{\mysp}{\g{sp}}
\DeclareMathOperator{\Pu}{Pu}

\DeclareMathOperator{\so}{\g{so}}

\DeclareMathOperator{\Spin}{Spin}
\DeclareMathOperator{\su}{\g{su}}

\DeclareMathOperator{\tr}{tr}

\newcommand{\C}{\ensuremath{\mathbb{C}}}
\newcommand{\Ca}{\ensuremath{\mathbb{O}}}

\newcommand{\eS}{\ensuremath{\mathrm{S}}}
\newcommand{\g}[1]{\operatorname{\ensuremath{\mathfrak{#1}}}}
\newcommand{\GL}{\ensuremath{\mathrm{GL}}}

\newcommand{\Lie}[1]{\operatorname{\sf{#1}}}
\newcommand{\oo}{\ast}

\newcommand{\R}{\ensuremath{\mathbb{R}}}
\renewcommand\Re{\mathrm{Re}}
\renewcommand{\H}{\ensuremath{\mathbb{H}}}

\begin{document}


\title{The bracket of the exceptional Lie algebra E8}
\author[A.\ Kollross]{Andreas Kollross}
\address{Institut f\"{u}r Geometrie und Topologie, Universit\"{a}t Stuttgart, Germany}
\email{kollross@mathematik.uni-stuttgart.de}
\date{\today}

\begin{abstract}
We obtain an explicit formula for the bracket of the exceptional simple Lie algebra~E8 based on triality and oct-octonions, following the Barton-Sudbery description of~E8. Furthermore, we provide descriptions of the subalgebras~E6 and~E7 and prove an explicit formula for the 26-dimensional irreducible representation of~F4.
\end{abstract}


\subjclass[2020]{17B25, 17A35, 22E60, 53C35}

\keywords{exceptional Lie algebra, octonions}

\maketitle


\section{Introduction}
\label{sec:intro}


The exceptional Lie groups and Lie algebras $\mathrm{G}_2$, $\mathrm{F}_4$, $\mathrm{E}_6$, $\mathrm{E}_7$, $\mathrm{E}_8$ were first described by Wilhelm Killing~\cite{killing,coleman}. The exceptional Lie groups have since made many appearances in geometry, topology, and physics. However, to this day, they still remain somewhat elusive and mysterious. 

Of the five exceptional Lie algebras $\g{g}_2$, $\g{f}_4$, $\g{e}_6$, $\g{e}_7$, $\g{e}_8$ the $248$-dimensional algebra~$\g{e}_8$ is the one with the highest dimension, and in many regards also the most difficult to deal with. For instance, since the lowest dimensional non-trivial representation is $248$-dimensional, it is not a good option to use a matrix representation of~$\g{e}_8$, an approach that is viable for~$\g{g}_2$, which has a $7$-dimensional faithful representation. For this and other reasons, the title of the article~\cite{garibaldi} refers to~$\mathrm{E}_8$ as the \emph{``most exceptional Lie group''} and it has even been argued in~\cite{garibaldi} that $\mathrm{E}_8$ could be called \emph{``the Monster of Lie theory''.}

Many different descriptions and constructions of~$\g{e}_8$ and other exceptional Lie algebras can already be found in the literature, as will be discussed below. It appears that exceptional Lie algebras and exceptional Lie groups have many different facets and thus there is no single one-size-fits-all approach to them, but different aspects of exceptional Lie groups lead to different constructions. Let us briefly review some known constructions of~$\g{e}_8$ and other Lie algebras, without aiming for completeness. 

\paragraph{\emph{Construction using a subalgebra.}} 
  One may construct a model for~$\g{g}=\g{e}_8$ and other Lie algebras by starting with a subalgebra~$\g{h}$ of~$\g{g}$.
  When restricted to~$\g{h}$, the adjoint representation of~$\g{g}$ splits as $\g{g} = \g{h} \oplus V$, where the real vector space~$V$ is an adjoint representation of. In order to define the Lie bracket on~$\g{g}$, one only has to define the bracket operation on~$V \times V$, since the bracket on~$\g{h} \times \g{h}$ and on~$\g{h} \times V$ is already given.
  This approach is employed in the articles~\cite{adams,draperetal,ms}.
  Various choices for the Lie algebra~$\g{h}$ lead to different models for~$\g{g}$. It appears to be most natural and convenient to choose the subalgebra~$\g{h}$ to be of full rank. 
  In the articles~\cite{adams} and ~\cite{ms}, the subalgebra~$\g{h}=\g{so}_{16}$ is used.
  In~\cite{draperetal} this procedure is applied to 14~different subalgebras of~$\g{e}_8$, but no explicit expressions for the Lie bracket are given. 
  
\paragraph{\emph{Construction using derivations.}} 
  It is well known that the compact Lie group~$\Lie{F}_4$ can be identified with the automorphism group of an Albert algebra and thus its Lie algebra can be constructed as the algebra of derivations of this Albert algebra. The Tits construction~\cite{tits}, see also~\cite{elduque,vinberg}, is an elegant and powerful generalization of this fact which also encompasses all the Lie algebras in Freudenthal's magic square. In~\cite{tits} and~\cite{vinberg}, formulae are given for the brackets of these Lie algebras. It is remarked in~\cite{wdm}: ``However, the definition of the Lie bracket given by Vinberg is not straightforward, and involves some quite complicated and unintuitive formulae.''. In this paper, we will give a more natural and simpler explicit formula for the bracket of~$\g{e}_8$.

\paragraph{\emph{Construction as a commutator algebra of matrix algebra.}} 
 A relatively new approach was recently used in~\cite{wdm}. It imitates the construction of classical Lie algebras arising from associative algebras using the commutator of matrices as a Lie bracket. However, since the octonions, which are used to define the matrix algebra, are non-associative, one cannot use the usual matrix multiplication and the commutator given by it, but the matrix multiplication has to be modified.

In this paper, we present an explicit formula for the Lie bracket of the compact real form of~$\g{e}_8$, which has the potential to drastically simplify calculations with~$\g{e}_8$ in certain situations. Our approach is based on the first of the three construction methods discussed above, namely, the Lie algebra~$\g{g} = \g{e}_8$ is built from its subalgebra $\g{h} = \g{so}_8 \oplus \g{so}_8$. As pointed out above, it only remains to define the bracket of~$\g{g} = \g{h} \oplus V$ on~$V \times V$ which we can do using a very a natural map, which is essentially just the multiplication in the algebra~$\Ca \otimes \Ca$. 
The appeal of this approach is that $\g{e}_8$ is built from scratch using little more than octonionic multiplication and the resulting formula~(\ref{eq:e8bracket}) is very concise, symmetric and intuitive.

Indeed, besides standard linear algebra operations, our formula only involves multiplication of Cayley numbers and the triality automorphism of~$\so(8)$ (which is also described using octonionic multiplication). 
Our description is based on the decomposition
\[
\g{e}_8 \cong \so(8) \oplus \so(8) \oplus (\Ca \otimes \Ca)^3,
\]
which is known as the \emph{Barton-Sudbery description} of~$\g{e}_8$, see~\cite{BaSu03}, by~Baez~\cite{Baez02}.

In~\cite{BaezWeb}, Baez writes:
\emph{``To emphasize the importance of triality, we can rewrite the Barton-Sudbery description of~$\g{e}_8$ as
\[
\g{e}_8 \cong \so(8) \oplus \so(8) \oplus (V_8 \otimes V_8) \oplus (S^+_8 \otimes S^+_8) \oplus (S^-_8 \otimes S^-_8)
\]
Here the Lie bracket is built from natural maps relating $\so(8)$ and its three 8-dimensional irreducible representations. In particular, $\so(8) \oplus \so(8)$ is a Lie subalgebra, and the first copy of $\so(8)$ acts on the first factor in $V_8 \otimes V_8$, $S_8^+ \otimes S_8^+$, and $S_8^- \otimes S_8^-$, while the second copy acts on the second factor in each of these. This has a pleasant resemblance to the triality description of $\g{f}_4$ [...]:
\[
\g{f}_4 \cong \so(8) \oplus V_8 \oplus S^+_8 \oplus S^-_8. \hbox{''}
\]
}
In the article~\cite{K18}, the author obtained an explicit bracket formula~(\ref{eq:f4bracket}) for~$\g{f}_4$ in order to study a geometric problem on the Cayley hyperbolic plane. This formula is based on the above-mentioned triality description of~$\g{f}_4$.
It appears that formulae of this type are particularly suitable for studying geometric problems involving Lie group actions and submanifold geometry,
see also~\cite{bdmn}, where the formula~(\ref{eq:f4bracket}) is applied to a geometric problem.

It turns out that there is a natural generalization~(\ref{eq:e8bracket}) of this formula for the bracket of~$\g{e}_8$. We obtain it by replacing~$\so(8)$ with~$\so(8) \oplus \so(8)$ and replacing the octonions~$\Ca$ with the real algebra~$\Ca \otimes \Ca$, the \emph{oct-octonions}. The main purpose of this paper is to prove that the formula~(\ref{eq:e8bracket}) actually defines a Lie algebra, i.e., satisfies the Jacobi identity, and this Lie algebra can be endowed with an invariant scalar product, showing that the resulting simple real Lie algebra is isomorphic to the compact real form of~$\g{e}_8$. This can also be regarded as a new existence proof for~$\g{e}_8$.

As applications, we provide descriptions of the subalgebras~$\g{e}_6$ and~$\g{e}_7$ of~$\g{e}_8$ and prove an explicit formula for the 26-dimensional irreducible representation of~$\g{f}_4$ at the end of the paper.


\section{Preliminaries}
\label{sec:prelim}


Recall that the \emph{octonions}, or \emph{Cayley numbers}, are an $8$-dimensional non-associative and non-commutative real division algebra~$\Ca$. We identify~$\Ca$ with~$\R^8$ by choosing a basis $e_0, \dots, e_7$ of~$\Ca$ such that $e := e_0$
is the multiplicative identity, we have $e^2_1 = \dots = e^2_7 = -e$ and the following multiplication rules:
\begin{equation}\label{eq:omult}
\begin{array}{llll}
  e_1e_2=e_3,\quad & e_1e_4=e_5,\quad & e_2e_4=e_6,\quad & e_3e_4=e_7, \\
  e_5e_3=e_6, & e_6e_1=e_7, & e_7e_2=e_5. &
\end{array}
\end{equation}
From these rules, the whole multiplication table of~$\Ca$ can be obtained using the fact that the alternative laws $(xx)y=x(xy)$ and $(xy)y=x(yy)$ hold for all $x,y \in \Ca$. 
We define \emph{octonionic conjugation} as the $\R$-linear map $\gamma \colon \Ca \to \Ca, a \mapsto \bar a$, given by
\[
\gamma(e_s) =
\left\{
  \begin{array}{ll}
    e_0, & \hbox{if~$s=0$;} \\
    -e_s, & \hbox{if~$s \ge 1$.}
  \end{array}
\right.
\]
We identify the real numbers~$\R$ with the subalgebra of~$\Ca$ spanned by~$e$ and define the \emph{real part} of an octonion~$x$ as the real number $\Re(x) := \frac12 (x + \bar x)$ and the \emph{pure part} as the octonion $\Pu(x) := \frac12(x - \bar x)$.
The \emph{norm} of an octonion~$a$ is defined as $|a| = \sqrt{a \bar a}$.
Recall that~$\Ca$ is a normed algebra, i.e., we have $|ab|=|a|\,|b|$ for all $a,b \in \Ca$. For $u,v \in \R^8 = \Ca$, we denote by $\langle u,v \rangle = \Re(u \bar v) = u^tv$ the standard scalar product of~$\R^8$.

Let~$\so(8)$ be the special orthogonal Lie algebra in~$8$~dimensions, which consists of the real skew-symmetric $8 \times 8$~matrices, where the bracket is given by the commutator of matrices. There is a map $\Ca \times \Ca \to \so(8)$, $(x,y) \mapsto x \wedge y$, defined by
\begin{equation}\label{eq:owedge}
x \wedge y := xy^t - yx^t,
\end{equation}
where, on the right hand side, the usual matrix multiplication is understood and $y^t$ is the row vector which is the transpose of~$y$. The elements $e_i \wedge e_j$, $0 \le i < j \le 7$, form a basis of~$\so(8)$.

For $a \in \Ca$, we define the $\R$-linear maps $L_a, R_a \colon \Ca \to \Ca$ as~$L_a(x) = ax$ and $R_a(x) = xa$, i.e., the left and right multiplication by~$a$.
Recall from~\cite{freudenthal}, \cite{murakami}, or~\cite{K18} that an automorphism of~$\so(8)$ of order three is given by the map
\begin{equation}\label{eq:lambdadef}
\lambda(a \wedge b) = \tfrac12 L_{\bar b} \circ L_{\bar a}  \hbox{~for~$a \in \Pu(\Ca), b \in \Ca$}.
\end{equation}
Recall furthermore that $\lambda^2 = \lambda \circ \lambda$ is given by
\begin{equation}\label{eq:lambda2def}
\lambda^2(a \wedge b) = \tfrac12 R_{\bar b} \circ R_{\bar a}  \hbox{~for~$a \in \Pu(\Ca), b \in \Ca$}.
\end{equation}
We define another automorphism $\kappa \in \Aut(\so(8))$ by
\[
\kappa(a \wedge b) = \bar a \wedge \bar b,
\]
or, equivalently, $\kappa(A)x=\overline{A \bar x}$ for $x \in \R^8$. Then we have
\[
\kappa^2 = \lambda^3 = 1 \qquad\hbox{and}\qquad \kappa \circ \lambda^2 = \lambda \circ \kappa,
\]
see~\cite{freudenthal}, \cite[\S2, Thm.~2]{murakami}.
In~\cite{K18} the following explicit expression was obtained for the Lie bracket of the simple compact Lie algebra~$\g{f}_4$.

\begin{theorem}\label{thm:f4bracket}
The binary operation on~$\so(8) \times \Ca^3$ defined by
\begin{align}\label{eq:f4bracket}
[(A,u,v,w),(B,x,y,z)] = (C,r,s,t)
\end{align}
where
\begin{align*}
C &= AB-BA-4u \wedge x - 4\lambda^2(v \wedge y) - 4\lambda(w \wedge z),\\
r &= Ax-Bu+\overline{vz}-\overline{yw},\\
s &= \lambda(A)y-\lambda(B)v+\overline{wx}-\overline{zu},\\
t &= \lambda^2(A)z-\lambda^2(B)w+\overline{uy}-\overline{xv},
\end{align*}
is $\R$-bilinear, skew symmetric and satisfies the Jacobi identity.
The real Lie algebra which is defined in this way is isomorphic to the simple compact Lie algebra~$\g{f}_4$.
\end{theorem}

In the following, we will obtain an analogous formula for the Lie bracket of~$\g{e}_8$.
This will be done by replacing $\so(8)$ with $\so(8) \oplus \so(8)$ and replacing $\Ca$ with the oct-octonions~$\Ca \otimes \Ca$. The formula for the Lie bracket of~$\g{e}_8$ that we obtain in this fashion resembles the above formula very closely, however, the natural action of~$\so(8)$ on~$\R^8$ is replaced by~(\ref{eq:so16act}), the triality automorphism~$\lambda$ is replaced by the simultaneous action of~$\lambda$ on both summands of~$\so(8) \oplus \so(8)$, which we will denote by the capital Greek letter~$\Lambda$ and the wedge product~$\wedge$ is replaced by the operator~$\curlywedge$ defined in~(\ref{eq:curlywedge}) below.


\section{\texorpdfstring{Oct-octonions and~$\so(8)\oplus\so(8)$}{Oct-octonions and so(8)x so(8)}}\label{sec:octonions}


We define the \emph{oct-octonions} to be the real algebra~$\Ca \otimes \Ca$, where the multiplication is given by
\[
(a \otimes b)\oo(c \otimes d) = ac \otimes bd.
\]
Using the identification $\Ca=\R^8$, we may identify $\Ca \otimes \Ca$ with $\R(8)$, the space of real $8 \times 8$~matrices. Indeed, viewing octonions as column vectors in~$\R^8$, we make the following identification for $x,y\in \Ca$:
\begin{equation}\label{eq:OOMAt}
x \otimes y = xy^t,
\end{equation}
where, on the right hand side, the usual matrix multiplication is understood and $y^t$ is the row vector which is the transpose of~$y$. We have
\[
(x \otimes y)^t = (xy^t)^t = yx^t = y \otimes x
\]
for $x,y \in \Ca$. Note that a notational ambiguity could arise from this identification, since now we have defined two different binary operations on~$\Ca \otimes \Ca = \R(8)$, oct-octonionic multiplication and the usual matrix multiplication. We avoid this ambiguity by writing $x \oo y$ for oct-octonionic multiplication, where $x,y\in \Ca \otimes \Ca$ and by denoting the matrix product by juxtaposition.

For $A = (P,Q) \in \so(8) \oplus \so(8)$ and $X \in \Ca \otimes \Ca = \R(8)$, we define
\begin{equation}\label{eq:so16act}
A.X := PX - XQ = PX+XQ^t,
\end{equation}
where the terms on the right hand side are defined by using the usual matrix multiplication, viewing $P,Q,X$ as real $8 \times 8$~matrices.
In particular, the Lie algebra $\so(8) \oplus \so(8)$ acts on~$\R^8 \otimes \R^8$ by the outer tensor product representation of the two standard representations of the two summands in~$\so(8) \oplus \so(8)$, i.e., we have
\begin{equation}\label{eq:PQact}
(P,Q).(r \otimes s) = (P,Q).sr^t = Prs^t+rs^tQ^t=(Pr) \otimes s + r \otimes (Qs)
\end{equation}
for $(P,Q) \in \so(8) \times \so(8)$ $X = r \otimes s$, $r,s \in \Ca$,
We furthermore define for $X,Y \in \R(8)$:
\begin{equation}\label{eq:curlywedge}
X \curlywedge Y := (XY^t-YX^t,X^tY-Y^tX) \in \so(8) \times \so(8).
\end{equation}
Note that in this definition, matrix multiplication is understood (instead of oct-octonionic multiplication).
Let $p,q,r,s \in \Ca$. We compute for later use:
\begin{align}
(p \otimes q) &\curlywedge (r \otimes s) = \nonumber \\
&=((p \otimes q)(r \otimes s)^t-(r \otimes s)(p \otimes q)^t,(p \otimes q)^t(r \otimes s)-(r \otimes s)^t(p \otimes q))= \nonumber \\
&=(pq^tsr^t-rs^tqp^t,qp^trs^t-sr^tpq^t)= \nonumber \\
&=(\langle q,s \rangle p \wedge r, \langle p,r \rangle q \wedge s). \label{eq:pqrs}
\end{align}

\begin{lemma}\label{lm:so16}
Let $(\so(8)\oplus\so(8)) \times \R(8)$ be equipped with the binary operation defined by
\[
[(A,X),(B,Y)] := (AB-BA - 4 X \curlywedge Y, A.Y-B.X),
\]
for $A,B \in \so(8)\oplus\so(8)$ and $X,Y \in \R(8)$.
Then the real algebra defined in this fashion is isomorphic to the Lie algebra~$\so(16)$.
\end{lemma}

\begin{proof}
We will check that an isomorphism is given by the map
\[
((P,Q),X) \mapsto
\begin{pmatrix}
  P & 2X \\
  -2X^t & Q \\
\end{pmatrix}.
\]
Indeed, let us compute the bracket
\begin{align*}
&\left[
\begin{pmatrix}
  P & 2X \\
  -2X^t & Q \\
\end{pmatrix}\right.,
\left.\begin{pmatrix}
  R & 2Y \\
  -2Y^t & S \\
\end{pmatrix}
\right] = \\
&\;\;\;\;=\begin{pmatrix}
  PR-RP-4XY^t+4YX^t & 2PY-2YQ-2RX+2XS \\
  2Y^tP-2QY^t-2X^tR+2SX^t & QS-SQ-4X^tY+4Y^tX \\
\end{pmatrix} =\\
&\;\;\;\;=\begin{pmatrix}
  [P,R]-4(XY^t-YX^t) & 2(P,Q).Y-2(R,S).X \\
  -(2(P,Q).Y-2(R,S).X)^t & [Q,S]-4(X^tY-Y^tX) \\
\end{pmatrix}
\end{align*}
This proves the statement of the lemma
\end{proof}
We define \emph{conjugation} $\Ca \otimes \Ca \to \Ca \otimes \Ca$, $x \mapsto \bar x$, of oct-octonions to be the $\R$-linear map given by
\[
\overline{a \otimes b} := \bar a \otimes \bar b
\]
for $a,b \in \Ca$.
Conjugation of oct-octonions is an anti-automorphism of $\Ca \otimes \Ca$, i.e., we have $\overline{xy}=\bar y \bar x$ for $x,y \in \Ca \otimes \Ca$, this follows from the fact that conjugation of octonions is an antiautomorphism of~$\Ca$.
We define an outer automorphism of order three of~$\so(8) \oplus \so(8)$ by
\begin{align}\label{eq:caplambda}
\Lambda \colon \so(8) \oplus \so(8) &\to \so(8) \oplus \so(8), \\
(P,Q) &\mapsto (\lambda(P),\lambda(Q)). \nonumber
\end{align}


\section{\texorpdfstring{The Lie algebra~$\mathfrak{e}_8$}{The Lie algebra e8}}\label{sec:e8}


\begin{theorem}\label{th:e8bracket}
The binary operation on~$\mathcal A = (\so(8)\oplus\so(8)) \times (\Ca \otimes \Ca)^3$ defined by
\begin{align}\label{eq:e8bracket}
[(A,u,v,w),(B,x,y,z)] = (C,r,s,t)
\end{align}
where
\begin{align*}
C &=[A,B]-4u \curlywedge x - 4\Lambda^2(v \curlywedge y) - 4\Lambda(w \curlywedge z),\\
r &= A.x-B.u+\overline{v \oo z}-\overline{y \oo w},\\
s &= \Lambda(A).y-\Lambda(B).v+\overline{w \oo x}-\overline{z \oo u},\\
t &= \Lambda^2(A).z-\Lambda^2(B).w+\overline{u \oo y}-\overline{x \oo v},
\end{align*}
and where $A,B \in \so(8)\oplus\so(8)$, $u,v,w,x,y,z \in \Ca \otimes \Ca$,
is $\R$-bilinear, skew symmetric and satisfies the Jacobi identity.
The real Lie algebra defined in this way is isomorphic to the Lie algebra of the compact exceptional simple Lie group of type~$\rm E_8$.
\end{theorem}

It follows directly by inspection that the binary operation~(\ref{eq:e8bracket}) is $\R$-bilinear and skew symmetric. To prove Theorem~\ref{th:e8bracket} we will verify that the Jacobi identity holds and that the resulting Lie algebra is a simple Lie algebra of dimension~$248$ which carries an $\ad$-invariant inner product.
Until we have completed the proof, we will write $\mathcal A$ for the real algebra defined by the operation~(\ref{eq:e8bracket}).

\begin{lemma}\label{lm:tau}
The linear map $\tau \colon \mathcal A \to \mathcal A$, given by
\[
\tau(A,x,y,z) = (\Lambda(A),y,z,x)
\]
is an automorphism of the real algebra~$\mathcal A$.
\end{lemma}

\begin{proof}
Follows directly by inspection of the formula~(\ref{eq:e8bracket}) and using the fact that $\Lambda$ is an automorphism of order three of~$\so(8)\oplus\so(8)$.
\end{proof}

We define an action of the Lie group~$\Spin(8) \times \Spin(8)$ on~$\R(8) = \Ca \otimes \Ca$ by
\[
\Theta(x) := (\theta_1,\theta_2)(u \otimes v) := \theta_1(u) \otimes \theta_2(v)
\]
for $\Theta=(\theta_1,\theta_2) \in \Spin(8) \times \Spin(8)$ and $x = u \otimes v \in \R(8)$ and where it is understood that an element of~$\Spin(8)$ acts by the standard representation of $\Spin(8)$ on~$\R^8$. Since $\Spin(8) \times \Spin(8)$ is simply connected, its automorphism group is canonically isomorphic to the automorphism group of~$\so(8) \oplus \so(8)$ and we denote the automorphism of~$\Spin(8) \times \Spin(8)$ given by $\Lambda \in \Aut(\so(8) \oplus \so(8))$ also by~$\Lambda$. Define an action of~$\Spin(8) \times \Spin(8)$ on $\mathcal A$ by
\[
 \Theta (A,x,y,z) = (\Ad_\Theta(A), \Theta(x), \Lambda(\Theta)(y), \Lambda^2(\Theta)(z) )
\]
for $\Theta \in \Spin(8) \times \Spin(8)$.

\begin{lemma}\label{lm:tauTheta}
We have
\[
\tau( \Theta(A,x,y,z) ) = \Lambda(\Theta)(\tau(A,x,y,z))
\]
for all $\Theta \in \Spin(8) \times \Spin(8)$, $(A,x,y,z) \in \mathcal A$.
\end{lemma}

\begin{proof}
We compute
\begin{align*}
\tau( \Theta(A,x,y,z) ) &= \tau(\Ad_\Theta(A), \Theta(x), \Lambda(\Theta)(y), \Lambda^2(\Theta)(z) )=\\
&=(\Ad_{\Lambda(\Theta)}(\Lambda(A)), \Lambda(\Theta)(y), \Lambda^2(\Theta)(z),  \Theta (x)) =\\
&=\Lambda(\Theta)(\Lambda(A), y, z, x) = \Lambda(\Theta)(\tau(A, x, y, z)). \qedhere
\end{align*}
\end{proof}

\begin{lemma}\label{lm:spin88aut}
The map \[(A,x,y,z) \mapsto \Theta (A,x,y,z)\] is an automorphism of the real algebra~$\mathcal A$ for any $\Theta \in \Spin(8) \times \Spin(8)$.
\end{lemma}

\begin{proof}
Let $\Theta = (\theta_1, \theta_2) \in \Spin(8) \times \Spin(8)$.
We first prove the automorphism property in some special cases, then we use linearity and the map $\tau$ to deduce the statement of the lemma.
Note that
\begin{align*}
&[\Theta(A,0,0,0),\Theta(B,0,0,0)] =
[(\Ad_\Theta(A),0,0,0),(\Ad_\Theta(B),0,0,0)] = \\
&=(\Ad_\Theta([A,B]),0,0,0) =
\Theta[(A,0,0,0),(B,0,0,0)].
\end{align*}
We compute
\begin{align*}
[\Theta&(A,0,0,0),\Theta(0,x,y,z)]
=[(\Ad_\Theta(A),0,0,0),(0, \Theta  (x), \Lambda(\Theta)  (y), \Lambda^2(\Theta) (z))] = \\
&= (0,\Ad_\Theta(A)(\Theta  (x)), \Lambda(\Ad_\Theta(A))(\Lambda(\Theta)(y)), \Lambda^2(\Ad_\Theta(A))(\Lambda^2(\Theta)(z))) =\\
&=(0,\Theta(Ax),\Lambda(\Theta)(\Lambda(A)y), \Lambda^2(\Theta)(\Lambda^2(A)z)) =\\
&=\Theta(0,Ax,\Lambda(A)y, \Lambda^2(A)z)=\Theta[(A,0,0,0),(0,x,y,z)].
\end{align*}
It follows from the principle of triality, see~\cite[\S2-3]{murakami}, that
\begin{equation}\label{eq:triality}
\theta(a)\lambda(\theta)(b) = \kappa\circ\lambda^2(\theta)(ab)
\end{equation}
for all $a,b \in \Ca$ and all $\theta \in \Spin(8)$.
We will use this in the following computation.
Let $u = p \otimes q$ and $y=r \otimes s$, where $p,q,r,s \in \Ca$.
\begin{align*}
[\Theta&(0,u,0,0),\Theta(0, 0, y, 0 )]
=[(0, \Theta(u),0,0),(0, 0, \Lambda(\Theta)(y), 0 )]=\\
&= (0,0,0,\overline{\Theta(u) \oo \Lambda(\Theta)(y)})=\\
&=(0,0,0,\overline{(\theta_1(p)\otimes\theta_2(q)) \oo (\lambda(\theta_1)(r)\otimes\lambda(\theta_2)(s))})=\\
&=(0,0,0,\overline{\theta_1(p)\lambda(\theta_1)(r) \otimes \theta_2(q)\lambda(\theta_2)(s)})=\\
&=(0,0,0,\overline{\kappa \circ \lambda^2(\theta_1)(pr) \otimes \kappa \circ \lambda^2(\theta_2)(qs)})=\\ &=(0,0,0,\overline{\kappa \circ \lambda^2(\theta_1)(pr)} \otimes \overline{\kappa \circ \lambda^2(\theta_2)(qs)})=\\
&=(0,0,0,\lambda^2(\theta_1)(\overline{pr}) \otimes \lambda^2(\theta_2)(\overline{qs}))=\\
&= (0,0,0,\Lambda^2(\Theta)(\overline{u \oo y}))=\Theta(0,0,0,\overline{u \oo y})=\Theta[(0,u,0,0),(0, 0,y, 0 )],
\end{align*}
where we have used that $\kappa(\theta)(x) = \overline{\theta(\bar x)}$ for all $\theta \in \Spin(8), x \in \Ca$.
Finally, assuming $u = p \otimes q$, $x = r \otimes s$ for $p,q,r,s \in \Ca$ and  $\Theta=(\theta_1,\theta_2)$ we obtain
\begin{align*}
[\Theta&(0,u,0,0),\Theta(0,x,0,0)]=[(0,\Theta(u),0,0),(0,\Theta(x),0,0)]=\\
&=[(0,\Theta(u),0,0),(0,\Theta(x),0,0)]=\\
&=(-4\Theta(u)\curlywedge\Theta(x),0,0,0)=\\
&=(-4(\theta_1pq^t\theta_2^t)\curlywedge(\theta_1rs^t\theta_2^t),0,0,0)=\\
&=(-4(\theta_1pq^t\theta_2^t\theta_2sr^t\theta_1^t-\theta_1rs^t\theta_2^t\theta_2qp^t\theta_1^t,\theta_2qp^t\theta_1^t\theta_1rs^t\theta_2^t-\theta_2sr^t\theta_1^t\theta_1pq^t\theta_2^t),0,0,0)=\\
&=(-4(\theta_1pq^tsr^t\theta_1^t-\theta_1rs^tqp^t\theta_1^t,\theta_2qp^trs^t\theta_2^t-\theta_2sr^tpq^t\theta_2^t),0,0,0)=\\
&=(-4\Ad_\Theta(pq^tsr^t-rs^tqp^t,qp^trs^t-sr^tpq^t),0,0,0)=\\
&=(-4\Ad_\Theta(u \curlywedge x),0,0,0)=\Theta([(0,u,0,0),(0,x,0,0)]).
\end{align*}
Using the skew symmetry and bilinearity of the bracket, the fact that $\tau$ is an automorphism of~$\mathcal A$ and  Lemma~\ref{lm:tauTheta}, the statement of the lemma now follows.
\end{proof}

\begin{lemma}\label{lm:e8jacobi}
The bracket operation on~$\mathcal A$ satisfies the Jacobi identity, i.e., we have
\begin{equation}\label{eq:Jacobi}
[\xi,[\eta,\zeta]]+[\eta,[\zeta,\xi]]+[\zeta,[\xi,\eta]]=0
\end{equation}
for all $\xi, \eta, \zeta \in \mathcal A$.
\end{lemma}

\begin{proof}
Since the left-hand side of the Jacobi identity~(\ref{eq:Jacobi}) is trilinear in $\xi, \eta, \zeta$, it suffices to show that it holds for all triples of vectors $(\xi,\eta,\zeta)$ where each one of the vectors $\xi,\eta,\zeta$ is an element from one of the four factors of $(\so(8)\oplus\so(8)) \times (\Ca\otimes\Ca) \times (\Ca\otimes\Ca) \times (\Ca\otimes\Ca)$.
Since the map~$\tau$ is an automorphism of~$\mathcal A$, we may furthermore assume the three vectors $\xi,\eta,\zeta$ are either from $(\so(8)\oplus\so(8)) \times (\Ca\otimes\Ca) \times (\Ca\otimes\Ca) \times \{0\}$ or from $\{0\} \times (\Ca\otimes\Ca) \times (\Ca\otimes\Ca) \times (\Ca\otimes\Ca)$.

\begin{enumerate}
\item
Assume first $\xi,\eta,\zeta \in (\so(8)\oplus\so(8)) \times (\Ca\otimes\Ca) \times (\Ca\otimes\Ca) \times \{0\}$. In the special cases where the vectors $\xi,\eta,\zeta$ are only taken from the first two factors or from the first and third factor, it follows from Lemma~\ref{lm:so16} (using the fact that $\tau$ is an automorphism of~$\mathcal A$ in the second case) that \ref{eq:Jacobi} holds. Therefore, we assume
\[
\xi = (A,0,0,0), \quad
\eta = (0,x,0,0), \quad
\zeta = (0,0,y,0).
\]
We compute the three summands on the left-hand side of~(\ref{eq:Jacobi}):
\begin{align*}
[(A,0,0,0),[(0,x,0,0),(0,0,y,0)]]&=(0,0,0,\Lambda^2(A).\overline{x \oo y}), \\
[(0,x,0,0),[(0,0,y,0),(A,0,0,0)]]&=(0,0,0,-\overline{x \oo (\Lambda(A).y)}),\\
[(0,0,y,0),[(A,0,0,0),(0,x,0,0)]]&=(0,0,0,-\overline{(A.x) \oo y}).
\end{align*}
Summing up the fourth components of these elements, and writing $A = (P,Q)$, $x = p \otimes q$, $y = r \otimes s$, where $P,Q \in \so(8)$, $p,q,r,s \in \Ca$, we get
\begin{align*}
\Lambda^2&(A).\overline{x \oo y} - \overline{x \oo (\Lambda(A).y)} - \overline{(A.x) \oo y} =\\
&=\Lambda^2(A).\overline{pr \otimes qs} - \overline{x \oo (\lambda(P)y-y\lambda(Q))} - \overline{(Px-xQ) \oo y} =\\
&=\lambda^2(P)\overline{pr} \otimes \overline{qs} + \overline{pr} \otimes \lambda^2(Q)\overline{qs}
 - \overline{(p \otimes q) \oo (\lambda(P)r \otimes s+r \otimes \lambda(Q)s)}
 -\\&\;\; - \overline{(Pp \otimes q + p \otimes Qq) \oo (r \otimes s)} =\\
&=\lambda^2(P)\overline{pr} \otimes \overline{qs} + \overline{pr} \otimes \lambda^2(Q)\overline{qs}
 - \overline{p\lambda(P)r \otimes qs+pr \otimes q\lambda(Q)s}
 -\\&\ - \overline{(Pp)r \otimes qs + pr \otimes (Qq)s)} =\\
&=\overline{(\kappa \circ \lambda^2(P) (pr)-p\lambda(P)r-(Pp)r) \otimes qs}+\\
 &\;\;+\overline{pr \otimes (\kappa \circ \lambda^2(Q)(qs)-q\lambda(Q)s-(Qq)s)}=\\
&=\overline{0 \otimes qs}+\overline{pr \otimes 0} = 0,
\end{align*}
where we have used the infinitesimal principle of triality, see~\cite[\S2, Thm.~1]{murakami}, which implies $(Au)v+u\lambda(A)v=\kappa \circ \lambda^2(A)(uv)$ for all $A \in \so(8)$ and all $u,v \in \Ca$.

\item
It now remains to study the case where the three elements~$\xi, \eta, \zeta$ from $\{0\} \times (\Ca\otimes\Ca) \times (\Ca\otimes\Ca) \times (\Ca\otimes\Ca)$. Again by using the automorphism~$\tau$, it suffices to study the two subcases where three elements are from three distinct factors in $(\Ca\otimes\Ca) \times (\Ca\otimes\Ca) \times (\Ca\otimes\Ca)$ or from only two distinct factors.
\begin{enumerate}
\item Let us assume
\[
\xi = (0,x,0,0), \quad
\eta = (0,0,y,0), \quad
\zeta = (0,0,0,z)
\]
where $x = e_i \otimes e_j$, $y = e_k \otimes e_\ell$, $z = t \otimes u$, $t,u \in \Ca$.
The group $\Spin(8) \times \Spin(8)$ acts as a group of automorphisms on~$\mathcal A$ by Lemma~\ref{lm:spin88aut} and we may use this action to assume $x=e \otimes e$ and $y=e \otimes e$, since $\Spin(8)$ acts transitively on the product of unit spheres~$\eS^7 \times \eS^7 \subset \R^8 \oplus \R^8$ by the sum of any two of its inequivalent irreducible 8-dimensional representations.
We compute, using~(\ref{eq:pqrs}), 
\begin{align*}
[(0,x,0,0)&,[(0,0,y,0),(0,0,0,z)]]= \\
&= [(0,e \otimes e,0,0), (0,\overline{t \otimes u},0,0)]=\\
&= (-4(e \otimes e) \curlywedge(\bar t \otimes \bar u),0,0,0)=\\
&= (-4(\langle e, \bar u \rangle e \wedge \bar t , \langle e,\bar t \rangle e \wedge \bar u),0,0,0), \\
[(0,0,y,0)&,[(0,0,0,z),(0,x,0,0)]] =\\
&=[(0,0,e \otimes e,0),(0,0,\overline{t \otimes u},0)]=\\
&=(-4\Lambda^2((e \otimes e) \curlywedge(\bar t \otimes \bar u)),0,0,0),\\
&=(-4\Lambda^2((\langle e, \bar u \rangle \bar t \wedge e, \langle e,\bar t \rangle \bar u \wedge e)),0,0,0)=\\
&=(-4(\langle e, \bar u \rangle \lambda^2(\bar t \wedge e), \langle e, \bar t \rangle \lambda^2(\bar u \wedge e)),0,0,0),\\
[(0,0,0,z)&,[(0,x,0,0),(0,0,y,0)]] =\\
&=[(0,0,0,t \otimes u),(0,0,0,e \otimes e)]=\\
&=(-4\Lambda((t \otimes u) \curlywedge (e \otimes e)),0,0,0)=\\
&=(-4\Lambda(\langle e, \bar u \rangle \bar t \wedge e, \langle e, \bar t \rangle \bar u \wedge e)),0,0,0)=\\
&=(-4(\langle e, \bar u \rangle \lambda(\bar t \wedge e), \langle e, \bar t \rangle \lambda(\bar u \wedge e))),0,0,0).
\end{align*}
The sum of these three elements is zero if
\begin{align*}
0&=
e \wedge \bar t  + \lambda^2(\bar t \wedge e) + \lambda(\bar t \wedge e) =
e \wedge \bar u  + \lambda^2(\bar u \wedge e) + \lambda(\bar u \wedge e).
\end{align*}
Note that, by linearity, we may assume $t,u \in \Pu(\Ca)$, since $\bar t \wedge e = 0$ if $t \in \R$ and $\bar u \wedge e = 0$ if $u \in \R$.
Then the above two equations are equivalent to
\begin{align*}
0 = e \wedge 2t  +  R_{t} +  L_{t} =
e \wedge 2u +  R_{u} +  L_{u}.
\end{align*}
It can be directly verified that $2t \wedge e = R_{t} +  L_{t}$ holds for $t \in \Pu(\Ca)$, see~\cite[Equation~(2.2)]{K18}.

\item Now consider\[
\xi = (0,x,0,0), \quad
\eta = (0,y,0,0), \quad
\zeta = (0,0,z,0).
\]
We compute
\begin{align*}
[(0,x,0,0)&,[(0,y,0,0),(0,0,z,0)]]=\\&=[(0,x,0,0),(0,0,0,\overline{y \oo z})]
=(0,0,-\bar x \oo (y \oo z),0), \\
[(0,y,0,0)&,[(0,0,z,0),(0,x,0,0)]]=\\&=[(0,y,0,0),(0,0,0,-\overline{x \oo z})]
=(0,0,\bar y \oo (x \oo z),0),\\
[(0,0,z,0)&,[(0,x,0,0),(0,y,0,0)]]=\\&=[(0,0,z,0),(-4 x \curlywedge y,0,0,0)]
=(0,0,4\Lambda(x \curlywedge y).z,0).
\end{align*}
Let $y= p \otimes q$. Using the $\Spin(8) {\times} \Spin(8)$-action, we may assume $x=z=e \otimes e$. Then the sum of the third components of the above three elements becomes
\begin{align*}
-y &+ \bar y +4\Lambda((e \otimes e) \curlywedge (p \otimes q)).(e \otimes e) =\\
&=-y + \bar y +4(\langle e ,q \rangle \lambda(e \wedge p),\langle e, p \rangle \lambda(e \wedge q)).(e \otimes e)=\\
&=-y + \bar y + 4 (\langle e ,q \rangle\lambda(e \wedge p)e \otimes  e + \langle e, p \rangle e \otimes \lambda(e \wedge q)e )=\\
&=-y + \bar y - 4 (\langle e ,q \rangle \tfrac12 L_{\Pu(p)}(e) \otimes  e + \langle e, p \rangle e \otimes \tfrac12 L_{\Pu(q)}(e) )=\\
&=-p \otimes q + \bar p \otimes \bar q  - 2( \langle e ,q \rangle \Pu(p) \otimes e + \langle e, p \rangle e \otimes \Pu(q) )=0.
\end{align*}
To see that the term in the last line is zero, one may use the identification~(\ref{eq:OOMAt}) to write the summands as matrices. Then $-p \otimes q + \bar p \otimes \bar q$ is a real $8 \times 8$~matrix which has non-zero entries only in the first row and the first column and so is $2( \langle e ,q \rangle \Pu(p) \otimes e + \langle e, p \rangle e \otimes \Pu(q) )$; it is easy to see that these two matrices are equal. 

\item Finally, we compute
\begin{align*}
[(0,x,0,0)&,[(0,0,y,0),(0,0,z,0)]]=\\&=[(0,x,0,0),(-4\Lambda^2(y \curlywedge z),0,0,0)]
 =(0,4\Lambda^2(y \curlywedge z).x,0,0), \\
[(0,0,y,0)&,[(0,0,z,0),(0,x,0,0)]]=\\&=[(0,0,y,0),(0,0,0,-\overline{x \oo z})]
 =(0, - (x \oo z) \oo \bar y,0,0),\\
[(0,0,z,0)&,[(0,x,0,0),(0,0,y,0)]]=\\&=[(0,0,z,0),(0,0,0,\overline{x \oo y})]
=(0,(x \oo y)\oo \bar z,0,0)
\end{align*}
We may again assume $x=z=e \otimes e$, $y=p \otimes q$. The sum of the second components of above three elements of~$\mathcal A$ is then
\begin{align*}
4&\Lambda^2 (y \curlywedge z).x - (x \oo z) \oo \bar y + (x \oo y)\oo \bar z =\\
&=4\Lambda^2 ((p \otimes q) \curlywedge (e \otimes e)).(e \otimes e) - \bar p \otimes \bar q + p \otimes q =\\
&=4\Lambda^2 (\langle q, e \rangle p \wedge e, \langle p, e \rangle q \wedge e).(e \otimes e) - \bar p \otimes \bar q + p \otimes q=\\
&=4(\langle q, e \rangle \lambda^2(p \wedge e), \langle p, e \rangle \lambda^2(q \wedge e)).(e \otimes e) - \bar p \otimes \bar q +p \otimes q=\\
&=4\langle q, e \rangle \lambda^2(p \wedge e)(e) \otimes e + 4\langle p, e \rangle e \otimes \lambda^2(q \wedge e)(e) - \bar p \otimes \bar q - + p \otimes q =\\
&=2\langle q, e \rangle R_{\Pu(p)}(e) \otimes e + 2\langle p, e \rangle e \otimes R_{\Pu(q)}(e) - \bar p \otimes \bar q + p \otimes q =\\
&=2\langle q, e \rangle \Pu(p) \otimes e + 2\langle p, e \rangle e \otimes \Pu(q) - \bar p \otimes \bar q + p \otimes q = 0,
\end{align*}
as above. 
\end{enumerate}
\end{enumerate}
We have shown that the Jacobi identity holds for the real algebra~$\mathcal A$.
\end{proof}

We define a scalar product on the Lie algebra~$\mathcal A$ by
\begin{equation}\label{eq:kill}
\langle ((A,B),u,v,w) , ((C,D),x,y,z) \rangle = 8\tr( ux^t+vy^t+wz^t)-\tr(AC)-\tr(BD).
\end{equation}
This scalar product is easily seen to be invariant under~$\tau$ and the $\Spin(8) \times \Spin(8)$-action.

\begin{lemma}\label{lm:killing}
The scalar product~(\ref{eq:kill}) on the Lie algebra~$\mathcal A$ is $\ad$-invariant.
\end{lemma}

\begin{proof}
Let $((A,B),u,v,w),((C,D),x,y,z) \in \mathcal A$ and let $P,Q \in \so(8)$.
We compute
\begin{align*}
&\langle [((A,B),u,v,w),((P,Q),0,0,0)],((C,D),x,y,z) \rangle = \\
&=\langle ([(A,B),(P,Q)],-(P,Q).u,-\Lambda(P,Q).v,-\Lambda^2(P,Q).w),((C,D),x,y,z) \rangle = \\
&=8\tr( (-Pu+uQ)x^t + (-\lambda(P)v+v\lambda(Q))y^t + (-\lambda^2(P)w+w\lambda^2(Q))z^t ) - \\
&\;\;-\tr(APC-PAC)-\tr(BQD-QBD)=\\
&=8\tr( u(Px-Qx)^t + v(\lambda(P)y-y\lambda(Q))^t + w(\lambda^2(P)z-z\lambda^2(Q))^t)  - \\
&\;\;-\tr(APC-ACP)-\tr(BQD-BDQ)=\\
&=\langle ((A,B),u,v,w),([(P,Q),(C,D)],(P,Q).x,\Lambda(P,Q).y,\Lambda^2(P,Q).z) \rangle =\\
&=\langle ((A,B),u,v,w),[((P,Q),0,0,0),((C,D),x,y,z)] \rangle.
\end{align*}

Furthermore, we have
\begin{align*}
&\langle [((A,B),u,v,w),( 0,e \otimes e,0,0)],((C,D),x,y,z) \rangle = \\
&=\langle (4(e \otimes e) \curlywedge u,(A,B).(e \otimes e),\bar w, -\bar v),((C,D),x,y,z) \rangle = \\
&=\langle ((4ee^tu^t-4uee^t, 4ee^tu-u^tee^t),Aee^t+ee^tB^t,\bar w, -\bar v),((C,D),x,y,z) \rangle = \\
&=8\tr(Aee^tx^t+ee^tB^tx^t+\bar wy^t-\bar vz^t) - \\ &\;\;-\tr(4ee^tu^tC-4uee^tC)-\tr(4ee^tuD-4u^tee^tD)=\\
&= \tr(8Aee^tx^t + 8B^tx^tee^t + 8Cuee^t + 8 Du^tee^t  -8\bar vz^t +8\bar wy^t) =\\
&=8\tr(-uee^tC^t-uDee^t -v\bar z^t +w\bar y^t) - \\
&\;\;-\tr(4Axee^t - 4Aee^tx^t) -\tr(4Bx^tee^t-4Bee^tx) =\\
&=\langle ((A,B),u,v,w),((4xee^t - 4ee^tx^t, 4x^tee^t-4ee^tx), -Cee^t-ee^tD^t, -\bar z, \bar y) \rangle =\\
&=\langle ((A,B),u,v,w),(4 x \curlywedge (e \otimes e), -(C,D).(e \otimes e), -\bar z, \bar y) \rangle =\\
&=\langle ((A,B),u,v,w),[(0,e \otimes e,0,0),((C,D),x,y,z)] \rangle.
\end{align*}

Here we have used the fact that $\tr(\bar x y^t) = \tr(x \bar y^t)$ for $x,y \in \Ca \otimes \Ca$.
Since the scalar product on~$\mathcal A$ is invariant under~$\tau$ and under the $\Spin(8) {\times} \Spin(8)$-action, the above calculations suffice to prove the $\ad$-invariance of the scalar product.
\end{proof}

\begin{proof}[Proof of Theorem~\ref{th:e8bracket}]
We have shown in Lemma~\ref{lm:e8jacobi} that the skew-symmetric operation of the real algebra~$\mathcal A$ satisfies the Jacobi identity. Therefore, $\mathcal A$ with the bracket defined in~(\ref{eq:e8bracket}) is a real Lie algebra. 
It can be seen from the bracket formula~(\ref{eq:e8bracket}) that if the adjoint representation of~$\mathcal A$ is restricted to the subalgebra $(\so(8) \oplus \so(8)) \times \{0\} \times \{0\} \times \{0\}$, then the resulting representation on~$\mathcal A$ is a direct sum of $5$~inequivalent irreducible modules. 
Thus any non-trivial ideal of~$\mathcal A$ is a direct sum of one or more of these $5$~summands. However, inspection of the formula~(\ref{eq:e8bracket}) shows that the only such ideal is~$\mathcal A$ itself and thus the Lie algebra~$\mathcal A$ is simple. We have~$\dim(\mathcal A) = 248$; using the formulae $\dim(\mathrm{A}_n) =n(n+2)$, $\dim(\mathrm{B}_n) = \dim(\mathrm{C}_n) =n(2n+1)$, $\dim(\mathrm{D}_n) =n(2n-1)$, $\dim(\mathrm{G}_2)=14$, $\dim(\mathrm{F}_4)=52$, $\dim(\mathrm{E}_6)=78$,  $\dim(\mathrm{E}_7)=133$, $\dim(\mathrm{E}_8)=248$, it is easy to see that any simple real Lie algebra of dimension~$248$ is of type~$\mathrm{E}_8$. There are three real forms of the complex simple Lie algebra of type~$\mathrm{E}_8$, the non-compact real forms~$\g{e}_{8(8)}$ and~$\g{e}_{8(-24)}$ and the compact form~$\g{e}_{8(-248)}$. Since we have shown in Lemma~\ref{lm:killing} that there exists an $\ad$-invariant scalar product on~$\mathcal A$, it follows that $\mathcal A$ is isomorphic to the compact real form~$\g{e}_{8(-248)}$.
\end{proof}

Using the duality of Riemannian symmetric spaces, we immediately obtain the bracket formula for the real noncompact simple Lie algebra~$\mathfrak{e}_{8(8)}$.

\begin{corollary}\label{cor:e88bracket}
The binary operation on~$(\so(8)\oplus\so(8)) \times (\Ca \otimes \Ca)^3$ defined by
\begin{align}
[(A,u,v,w),(B,x,y,z)] = (C^*,r^*,s,t)
\end{align}
where
\begin{align*}
C^* &=[A,B] - 4u \curlywedge x + 4\Lambda^2(v \curlywedge y) + 4\Lambda(w \curlywedge z),\\
r^* &= A.x-B.u-\overline{v \oo z}+\overline{y \oo w},\\
s &= \Lambda(A).y-\Lambda(B).v+\overline{w \oo x}-\overline{z \oo u},\\
t &= \Lambda^2(A).z-\Lambda^2(B).w+\overline{u \oo y}-\overline{x \oo v},
\end{align*}
where $A,B \in \so(8)\oplus\so(8)$, $u,v,w,x,y,z \in \Ca \otimes \Ca$,
defines the bracket operation of a real Lie algebra isomorphic to the noncompact exceptional simple algebra~$\mathfrak{e}_{8(8)}$.
\end{corollary}

\begin{proof}
Let $\mathfrak{g} = \mathfrak{e}_8$, the compact form.
It is easy to see from the bracket formula~(\ref{eq:e8bracket}) and Lemma~\ref{lm:so16} that the map $(A,u,v,w) \mapsto (A,u,-v,-w)$ is an involutive automorphism of~$\mathfrak{g}$, corresponding to the Riemannian symmetric space ${\rm E}_8 / \Spin(16)$. Let $\mathfrak{g} = \mathfrak{k} \oplus \mathfrak{p}$ be the decomposition into the $(+1)$-eigenspace and $(-1)$-eigenspace. It follows that $\mathfrak{g} = \mathfrak{k} \oplus i\mathfrak{p} \subset \mathfrak{g} \otimes \C$ is the Lie algebra of the isometry group of the dual symmetric space ${\rm E}_{8(8)} / \Spin(16)$, see~\cite{helgason}. Its bracket is given by the above formula.
\end{proof}


\section{\texorpdfstring{The Lie algebra~$\mathfrak{e}_{6}$}{The Lie algebra e6}}


Let $\g{t}^2$ be the abelian subalgebra of~$\so(8)$ spanned by the elements $R_{e_1}$, $L_{e_1}$.
Then by equation~(2.2) in~\cite{K18}, we have $T_{e_1} = 2 e_1 \wedge e_0 = R_{e_1}+L_{e_1} \in \g t^2$ and it follows that $\lambda(T_{e_1}) = 2\lambda(e_1 \wedge e_0) = L_{\bar e_0} \circ L_{\bar e_1} = -L_{e_1}$ and $\lambda^2(T_{e_1}) = 2\lambda^2(e_1 \wedge e_0) = R_{\bar e_0} \circ R_{\bar e_1} = -R_{e_1}$. Hence $\g{t}^2$ is invariant under the action of the triality automorphism~$\lambda$.
Furthermore, it follows that 
\[
\lambda(L_{e_1})=R_{e_1}, \qquad \lambda^2(L_{e_1})=-T_{e_1}, \qquad \lambda(R_{e_1})=-T_{e_1}, \qquad \lambda^2(R_{e_1})=L_{e_1}.
\]
We define $\C \subset \Ca$ as the linear subspace spanned by $e_0$ and~$e_1$.
Define the subspace
\[
\g{e}_6 := (\so(8) \times \g{t}^2) \times (\Ca \otimes \C)^3
\]
of~$\g{e}_8$.

\begin{lemma}
The linear subspace~$\g{e}_6$ of~$\g{e}_8$ is a subalgebra isomorphic to the Lie algebra of the compact exceptional simple Lie group of type~$\rm E_6$.
\end{lemma}

\begin{proof}
It is clear that $\so(8) \times \g{t}^2$ is a subalgebra of~$\so(8) \times \so(8)$.
Let $A \in \so(8)$, $H \in \g{t}^2$, $x,y,z \in \Ca$, $a,b,c \in \C$.
Let us compute the following bracket: 
\begin{align}\label{eq:AHV}
[&((A,H),0,0,0),((0,0),x \otimes a, y \otimes b, z \otimes c) = \\ \nonumber 
&=((0,0),(A,H).(x \otimes a), \Lambda(A,H).(y \otimes b), \Lambda^2(A,H).(z \otimes c)) =\\ \nonumber
&=((0,0),A(x) \otimes a+x \otimes H(a), \lambda(A)(y) \otimes b+y \otimes \lambda(H)(b), \lambda^2(A)(z) \otimes c+z \otimes \lambda^2(H)(c)).
\end{align}
Since $R_{e_1}$ and $L_{e_1}$ leave $\C$ invariant, this bracket lies in~$\g{e}_6$.
Let us compute some more brackets. We have
\begin{align}\label{eq:xaoyb}
[((0,0),x \otimes a,0,0)&,((0,0),0, y \otimes b, 0) = \\  \nonumber
&=((0,0),0, 0, (y \oo z) \otimes ab) \in \g{e}_6,
\end{align}
and
\begin{align}\label{eq:xayb}
[((0,0),x \otimes a,0,0)&,((0,0),y \otimes b, 0, 0) = \\ \nonumber
&=(-4 (x \otimes a) \curlywedge (y \otimes b) ,0, 0, 0) =\\ \nonumber
&=(-4 (\langle a,b \rangle x \wedge y, \langle x,y \rangle a \wedge b) ,0, 0, 0),
\end{align}
where we have used~(\ref{eq:pqrs}).
Let us compute $a \wedge b = (a_0 e_0 + a_1 e_1) \wedge (b_0 e_0 + b_1 e_1) = (a_0 b_1 - a_1 b_0) e_0 \wedge e_1 = - \frac 12\Im(a \bar b) T_{e_1}$, showing that this bracket lies in~$\g{e}_6$. Finally, since the map $((A,B),x,y,z) \mapsto (\Lambda(A,B),y,z,x)$ is an automorphism of~$\g{e}_8$ by Lemma~\ref{lm:spin88aut} which restricts to~$\g{e}_6$, the above calculations show that~$\g{e}_6$ is a subalgebra. We have constructed a subalgebra of~$\g{e}_8$ which contains~$\g{f}_4$ as a subalgebra. Let $V$ be the 26-dimensional orthogonal complement of~$\g{f}_4$ in~$\g{e}_6$.
It now suffices to observe that~$\g{f}_4$ acts nontrivially on~$V$ by restriction of the adjoint representation of~$\g{e}_6$ to show that this representation is irreducible, since the lowest-dimensional nontrivial real representation of~$\g{f}_4$ is its 26-dimensional irreducible representation. One can check immediately that $V$ is not an ideal of~$\g{e}_6$. This shows that $\g{e}_6$ is a 78-dimensional simple subalgebra of~$\g{e}_8$ containing~$\g{f}_4$ as a subalgebra and from this we conclude that it is isomorphic to the Lie algebra of the compact exceptional simple Lie group of type~$\rm E_6$.
\end{proof}

We have shown that $\g{f}_4$ acts by its 26-dimensional irreducible representation on its orthogonal complement~$V$ in~$\g{e}_6$.
We deduce the following explicit formula for this action.

\begin{lemma}
Let us identify the orthogonal complement $V = \g{t}^2 \times (\Ca \otimes e_1)^3$ of~$\g{f}_4$ in~$\g{e}_6$ with $\R^2 \times \Ca^3$ as follows:
$
(t,s,u,v,w) \mapsto (t L_{e_1} + s R_{e_1}, u \otimes e_1, v \otimes e_1, w \otimes e_1). 
$
Let $(A,x,y,z) \in \g{f}_4$, $(t,s,x,y,z) \in V$.
Then the Lie bracket of these elements in~$\g{e}_6$ is given by
\begin{align}\label{eq:f4action}
[(A,u,v,w),(t,s,x,y,z)] = (\tilde{t},\tilde{s},\tilde{x},\tilde{y},\tilde{z}),
\end{align}
where
\begin{align*}
\tilde{t} &= 2 \langle u,x \rangle -2 \langle w,z \rangle, \\
\tilde{s} &= 2 \langle u,x \rangle -2 \langle v,y \rangle, \\
\tilde{x} &= Ax+(t+s)u+\overline{vz}-\overline{yw},\\
\tilde{y} &= \lambda(A)y-(t-2s)v+\overline{wx}-\overline{zu},\\
\tilde{z} &= \lambda^2(A)z-(s-2t)w+\overline{uy}-\overline{xv}.
\end{align*}
\end{lemma}

\begin{proof}
Let us compute the components of the right hand side of~(\ref{eq:f4action}).
Using (\ref{eq:xayb}), we get
\begin{align*}
[((0,0),u \otimes e_0,0,0)&,((0,0),x \otimes e_1, 0, 0) = \\
&=(-4 (\langle e_0,e_1 \rangle u \wedge x, \langle u,x \rangle e_0 \wedge e_1) ,0, 0, 0) 
=(0, 2 \langle u,x \rangle T_{e_1}) ,0, 0, 0).
\end{align*}
From this, we obtain
\begin{align*}
[((0,0),0,v \otimes e_0,0)&,((0,0),0,y \otimes e_1, 0) 
=(0, 2 \langle v,y \rangle \lambda^2(T_{e_1})) ,0, 0, 0)
\end{align*}
and
\begin{align*}
[((0,0),0,0,w \otimes e_0)&,((0,0),0,0,z \otimes e_1) 
=(0, 2 \langle w,z \rangle \lambda(T_{e_1})) ,0, 0, 0).
\end{align*}
Keeping in mind that $\lambda(T_{e_1})=-L_{e_1}$, $\lambda^2(T_{e_1})=-R_{e_1}$, and $T_{e_1}=R_{e_1}+L_{e_1}$, we obtain the expression for~$\tilde{t}$ and~$\tilde{s}$ in the statement for the lemma.
It remains to determine $\tilde x$, $\tilde y$, $\tilde z$.
We will use~(\ref{eq:AHV}). Let $H = tL_{e_1} + sR_{e_1}$, where $t,s \in \R$ and $x \in \Ca$.
Then
\begin{align*}
[((&0,H),0,0,0),((0,0),u \otimes e_1, 0, 0)] = ((0,0),(0,H).(u \otimes e_1), 0, 0)  = \\ &= ((0,0),u \otimes H(e_1), 0, 0) =((0,0),(-t-s)u, 0, 0),
\end{align*}
since $H(e_1) = tL_{e_1}(e_1) + sR_{e_1}(e_1) = -t-s$.
Similarly, $\lambda(H)(e_1) = \lambda(tL_{e_1} + sR_{e_1})(e_1) 
= tR_{e_1}(e_1) - sT_{e_1}(e_1) = tR_{e_1}(e_1) - sR_{e_1}(e_1) - sL_{e_1}(e_1) = -t +2s$
and $\lambda^2(H)(e_1) = \lambda^2(tL_{e_1} + sR_{e_1})(e_1) 
= -tT_{e_1}(e_1) + sL_{e_1}(e_1) = -tL_{e_1}(e_1)-tR_{e_1}(e_1)  + sL_{e_1}(e_1) =2t-s$.
This yields the second summands in the formulas for $\tilde x$, $\tilde y$, $\tilde z$.
The others can be read off from~(\ref{eq:e8bracket}).
\end{proof}


\section{\texorpdfstring{The Lie algebra~$\mathfrak{e}_{7}$}{The Lie algebra e7}}


Define $\H$ to be the subalgebra of~$\Ca$ spanned by $e_0=e,e_1,e_2,e_3$ and define the subalgebra
\begin{equation}\label{eq:so4def}
\so(4) :=
\left\{ \left(
\begin{array}{c|c}
  A &  \\ \hline
   & 0 \\
\end{array} \right)
\colon
A \in \R^{4 \times 4}, A^t = - A
\right\} \subset \so(8).
\end{equation}
Let $U_0 = \{0\} \times \H \times \{0\} \times \{0\}$.
Then it is easy to check that $\so(4)+U_0$ is a subalgebra of~$\so(8)$ isomorphic to~$\so(5) \cong \mysp(2)$ and we have $[U_0,U_0]=\so(4)$.
Define the subspaces $U_1 = \{0\} \times \{0\} \times \H \times \{0\}$ and $U_2 = \{0\} \times \{0\} \times \{0\} \times \H$ of~$\g{f}_4$. 
We have $[U_1,U_1]=\lambda^2(\so(4))$ and $[U_2,U_2]=\lambda(\so(4))$.
Now define
\[
\mysp(1)^3 := \so(4) + \lambda(\so(4)) + \lambda^2(\so(4)).
\]
It was shown in~\cite{K18} that this subspace is a subalgebra of~$\so(8)$. 
The representations $\lambda|_{\so(4)}$ and $\lambda^2|_{\so(4)}$ leave the subspaces $\H$ and $\H e_4$ of~$\Ca$ invariant and act nontrivially on both of them. 
Let $\varrho \colon \so(4) \to \GL(\H e_4)$ be the representation defined by $\varrho(A)(x) = \lambda(A)(x)$ and let $\varphi \colon \so(4) \to \GL(\H e_4)$ be the representation defined by $\varphi(A)(x) = \lambda^2(A)(x)$.
Since we have  $L_a|_{\H e_4} = -R_a|_{\H e_4}$ for $a \in \Pu(\H)$, it follows that the images of $\varrho$ and $\varphi$ are the same. Therefore, we have
\begin{equation}\label{eq:sp13descr}
\g{sp}(1)^3 =
\left\{ \left(
\begin{array}{c|c}
  A &  \\ \hline
   & \varrho(B) \\
\end{array} \right) 
\colon
A \in \R^{4 \times 4}, A^t = - A, \; B \in \su(2)
\right\},
\end{equation}
where we denote by $\su(2)$ the simple ideal of~$\so(4)$ which is not contained in the kernel of~$\varrho$. 
It follows that $\g{sp}(1)^3$ is a subalgebra of~$\so(8)$ isomorphic to $\mysp(1) \oplus \mysp(1) \oplus \mysp(1)$. By definition, we have $\lambda(\g{sp}(1)^3) = \g{sp}(1)^3$.

We define the following linear subspace of~$\mathfrak{e}_{8}$:
\[
\mathfrak{e}_{7} := \so(8) \oplus \mysp(1)^3 \oplus (\Ca \otimes \H)^3.
\]

\begin{lemma}
The linear subspace $\mathfrak{e}_{7}$ is a subalgebra of~$\mathfrak{e}_{8}$, which is isomorphic to the Lie algebra of the exceptional compact Lie group~$\mathsf{E}_7$.
\end{lemma}

\begin{proof}
We first note that the action~(\ref{eq:PQact}) restricted to $(\so(8) \oplus \mysp(1)^3) \times (\Ca \otimes \H)$ obviously leaves $(\Ca \otimes \H)$ invariant. 
Since $\g{sp}(1)^3$ is invariant under~$\lambda$, it follows that $\so(8) \oplus \mysp(1)^3$ is invariant under~$\Lambda$.
This in turn shows that taking the bracket of elements in~$\so(8) \oplus \mysp(1)^3$ with elements in~$(\Ca \otimes \H)^3$ results in elements of~$(\Ca \otimes \H)^3$.
Since $\H \wedge \H \subseteq \so(4)$, formula~(\ref{eq:pqrs}), together with the invariance of~$\mysp(1)^3$ under~$\lambda$, shows that the bracket of two elements in~$(\Ca \otimes \H)^3$ is contained in~$\mathfrak{e}_{8}$.

Restricting the adjoint representation of~$\mathfrak{e}_{7}$ to the subalgebra~$\so(8) \oplus \mysp(1)^3$, it splits into $7$~mutually inequivalent irreducible submodules; it can be easily checked that none of them is an ideal; hence the subalgebra~$\mathfrak{e}_{7}$ is simple.
Being a subalgebra of~$\mathfrak{e}_{8}$, the Lie algebra~$\mathfrak{e}_{7}$ is compact; since it is $133$-dimensional, it follows that $\mathfrak{e}_{7}$ is isomorphic to the Lie algebra of exceptional compact Lie group~~$\mathsf{E}_7$.
\end{proof}


\end{document}